\documentclass[11pt,oneside,english]{article}
\usepackage[T1]{fontenc}
\usepackage[latin9]{inputenc}
\usepackage{geometry}
\usepackage{amstext}
\usepackage{amsthm}
\usepackage{amsmath,amsfonts,amsthm,amssymb,mathrsfs,dsfont} 
\usepackage{mathtools}
\usepackage{graphicx}
\usepackage[colorinlistoftodos,textsize=scriptsize,backgroundcolor=orange!20]{todonotes}
\usepackage[colorlinks=true, allcolors=blue]{hyperref}
\usepackage[capitalize,nameinlink]{cleveref} 
\crefname{equation}{}{}

\usepackage{verbatim}
\usepackage{enumerate}
\allowdisplaybreaks

\makeatletter
\newtheorem{theorem}{Theorem}[section]

\newcommand{\theoremname}{testing}

\newtheorem{lemma}[theorem]{Lemma}

\newtheorem{proposition}[theorem]{Proposition}

\newtheorem{conjecture}[theorem]{Conjecture}
\newtheorem*{question*}{Question}

\theoremstyle{definition}
\newtheorem{definition}[theorem]{Definition}

\theoremstyle{remark}
\newtheorem{remark}[theorem]{Remark}

\theoremstyle{plain}

\makeatother

\usepackage{babel}

\providecommand{\theoremname}{Theorem}

\DeclareMathOperator{\Span}{span}
\DeclareMathOperator{\rank}{rank}

\title{On the number of Hadamard matrices via anti-concentration}
\author{Asaf Ferber \thanks{Massachusetts Institute of Technology. Department of Mathematics. Email: {\tt ferbera@mit.edu}. Research is partially supported by NSF DMS-6935855.} \and 
Vishesh Jain\thanks{Massachusetts Institute of Technology. Department of Mathematics. Email: {\tt visheshj@mit.edu}. Research is partially supported by NSF CCF 1665252 and DMS-1737944 and ONR N00014-17-1-2598.} \and
Yufei Zhao \thanks{Massachusetts Institute of Technology. Department of Mathematics. Email: {\tt yufeiz@mit.edu}. Research is partially supported by NSF DMS-1362326 and DMS-1764176.}}

\date{}

\begin{document}
\maketitle
\global\long\def\R{\mathbb{R}}

\global\long\def\S{\mathbb{S}}

\global\long\def\Z{\mathbb{Z}}

\global\long\def\C{\mathbb{C}}

\global\long\def\Q{\mathbb{Q}}

\global\long\def\N{\mathbb{N}}

\global\long\def\P{\mathbb{P}}

\global\long\def\F{\mathbb{F}}

\global\long\def\U{\mathcal{U}}

\global\long\def\V{\mathcal{V}}

\global\long\def\E{\mathbb{E}}

\global\long\def\bad{\text{BAD}}

\global\long\def\vbad{\text{VBAD}}

\global\long\def\KL{\text{KL}}

\global\long\def\comp{\text{Comp}}

\global\long\def\rk{\text{rank}}

\global\long\def\A{\mathcal{A}}

\global\long\def\L{\mathcal{L}}

\begin{abstract}
Many problems in combinatorial linear algebra require upper bounds on the number of solutions to an underdetermined system of linear equations $Ax = b$, where the coordinates of the vector $x$ are restricted to take values in some small subset (e.g. $\{\pm 1\}$) of the underlying field. The classical ways of bounding this quantity are to use either a rank bound observation due to Odlyzko or a vector anti-concentration inequality due to Hal\'asz. The former gives a stronger conclusion except when the number of equations is significantly smaller than the number of variables; even in such situations, the hypotheses of Hal\'asz's inequality are quite hard to verify in practice. In this paper, using a novel approach to the anti-concentration problem for vector sums, we obtain new Hal\'asz-type inequalities which beat the Odlyzko bound even in settings where the number of equations is comparable to the number of variables. In addition to being stronger, our inequalities have hypotheses which are considerably easier to verify. We present two applications of our inequalities to combinatorial (random) matrix theory: (i) we obtain the first non-trivial upper bound on the number of $n\times n$ Hadamard matrices, and (ii) we improve a recent bound of Deneanu and Vu on the probability of normality of a random $\{\pm 1\}$ matrix.     
\end{abstract}


\section{Introduction}


\subsection{The number of Hadamard matrices}
A square matrix $H$ of order $n$ whose entries are $\{\pm{1}\}$ is called a \emph{Hadamard matrix of order $n$} if its rows are pairwise orthogonal i.e. if $ HH^{T} = nI_{n}$. 
They are named after Jacques Hadamard, who studied them in connection with his maximal determinant problem. Specifically, Hadamard asked for the maximum value of the determinant of any $n \times n$ square matrix all of whose entries are bounded in absolute value by $1$. He proved \cite{hadamard1893resolution} that the  value of the determinant of such matrices cannot exceed $n^{n/2}$. Moreover, he showed that Hadamard matrices are the only ones that can attain this bound. Since their introduction, Hadamard matrices have been the focus of considerable attention from many different communities -- coding theory, design theory, statistical inference, and signal processing to name a few. We refer the reader to the surveys \cite{hedayat1978hadamard, seberry1992hadamard} and the books \cite{agaian2006hadamard,horadam2012hadamard} for a comprehensive account of Hadamard matrices and their numerous applications.  

Hadamard matrices of order $1$ and $2$ are trivial to construct, and it is quite easy to see, by considering the first few rows, that every other Hadamard matrix (if exists) must be of order $4m$ for some $m\in \N$. Whereas Hadamard matrices of infinitely many orders have been constructed, the question of whether one of order $4m$ exists for every $m\in \N$ is the most important open question on this topic, and remains wide open.

\begin{conjecture}[The Hadamard conjecture, \cite{paley1933orthogonal}]
There exists a Hadamard matrix of order $4m$ for every $m\in \N$. 
\end{conjecture}

In this paper, we study the question of \emph{how many} Hadamard matrices of order $n=4m$ could possibly exist for a given $m\in \N$. Let us denote this number by $H(n)$. Note that if a single Hadamard matrix of order $n$ exists, then we immediately get at least $(n!)^{2}$ distinct Hadamard matrices by permuting all the rows and columns. Thus, if the Hadamard conjecture is true, then $H(n) = 2^{\Omega( n\log{n})}$ for every $n=4m, m\in \N$. On the other hand, the bound $H(n) \leq 2^{\binom{n+1}{2}}$ is quite easy to obtain, as we will discuss in the next subsection. 

This bound also appeared in the work of de Launey and Levin \cite{de2010fourier} on the enumeration of partial Hadamard matrices (i.e. $k\times 4m$ matrices whose rows are pairwise orthogonal, in the limit as $m\to \infty$) using Fourier analytic techniques; notably, while they were able to get a very precise answer to this problem (up to an overall $(1+o(1))$ multiplicative factor), their techniques still did not help them to obtain anything better than the essentially trivial bound for the case of square Hadamard matrices. As our first main result, we give the only known non-trivial upper bound on the number of square Hadamard matrices.  
\begin{theorem}
\label{thm:hadamard}
There exists an absolute constant $c_H > 0$ such that $H(n)\leq 2^{\frac{(1-c_H)n^2}{2}}$ for all sufficiently large $n$ that is a multiple of 4. 
\end{theorem}
\begin{remark}
In our proof of the above theorem, we have focused on the simplicity and clarity of presentation and have made no attempt to optimize this constant, since our proof cannot give a value of $c_H$ larger than (say) $\frac{1}{2}$ whereas we believe that the correct value of $c_H$ should be close (as a function of $n$) to $1$. 
\end{remark}

\begin{conjecture}
For any $n=4m, m\in \N$, $H(n) = 2^{O(n\log{n})}$. 
\end{conjecture}
We believe that proving a bound of the form $H(n) = 2^{o(n^2)}$ will already be very interesting, and will likely require new ideas. 

\subsubsection{The approach}
\label{subsec-approach}
We now discuss the proof of the trivial upper bound $H(n) \leq 2^{\binom{n+1}{2}}$. The starting point is the following classical (and almost trivial to prove) observation due to Odlyzko. 
\begin{lemma}[Odlyzko, \cite{odlyzko1988subspaces}] \label{odlyzko}Let $W$ be a $d$-dimensional subspace of $\R^n$. Then, $|W\cap \{\pm1\}^n|\leq 2^d$.   
\end{lemma}

\begin{proof}[Sketch]
As $W$ is a $d$-dimensional space, it depends only on $d$ coordinates. Therefore, it spans at most $2^d$ vectors with entries from $\{\pm1\}$.
\end{proof}
The bound $H(n) \leq 2^{\binom{n+1}{2}}$ is now immediate. Indeed, we construct the matrices row by row, and note that by the orthogonality of the rows, the first $k$ rows span a subspace of dimension $k$ to which the remaining rows are orthogonal. In particular, once the first $k$ rows have been selected, the $(k+1)^{st}$ row lies in a specified subspace of dimension $n-k$ (the orthogonal complement of the vector space spanned by the first $k$ (linearly independent) rows), and hence, by Lemma \ref{odlyzko}, is one of at most $2^{n-k}$ vectors. It follows that $H(n) \leq \prod_{i=0}^{n-1}2^{n-i} = 2^{\binom{n+1}{2}}$.   

The weak point in the above proof is the following -- while Odlyzko's bound is tight in general, we should expect it to be far from the truth in the average case. Indeed, working with vectors in $\{0,1\}^{n}$ for the moment, note that a subspace of dimension $k$ spanned by vectors in $\{0,1\}^{n}$ has exactly $2^{n-k}$ vectors in $\{0,1\}^{n}$ orthogonal to it \emph{viewed as elements of $\F_{2}^{n}$}. However, typically, the inner products will take on many values in $2\Z \setminus \{0\}$ so that many of these vectors will not be orthogonal viewed as elements of $\R^{n}$. 

The study of the difference between the Odlyzko bound and how many $\{\pm 1\}^{n}$ vectors a subspace actually contains has been very fruitful in discrete random matrix theory, particularly for the outstanding problem of determining the probability of singularity of random $\{\pm 1\}$ matrices. Following Kahn, Koml\'os and Szemer\'edi \cite{kahn1995probability}, Tao and Vu \cite{tao2007singularity} isolated the following notion.

\begin{definition}[Combinatorial dimension] The \emph{combinatorial dimension} of a subspace $W$ in $\R^n$, denoted by $d_{\pm}(W)$, is defined to be smallest real number such that 
$$\left|W \cap \{\pm 1\}^{n}\right| \leq 2^{d_{\pm}(W)}.$$
\end{definition}

Thus, Odlyzko's lemma says that for any subspace $W$, its combinatorial dimension is no more than its dimension. However, improving on another result of Odlyzko \cite{odlyzko1988subspaces},
Kahn, Koml\'os and Szemer\'edi showed that this bound is very loose for \emph{typical subspaces spanned by $\{\pm 1\}^{n}$ vectors}: 
\begin{theorem}[Kahn-Koml\'os-Szemer\'edi, \cite{kahn1995probability}]
There exists a constant $C>0$ such that if $r\leq n-C$, and if $v_{1},\dots,v_{r}$
are chosen independently and uniformly from $\{\pm1\}^{n}$, then
\[
\Pr\left[d_{\pm}(\Span\{v_{1},\dots,v_{r}\}) > \log_{2}(2r) \right]=(1+o(1))4{r \choose 3}\left(\frac{3}{4}\right)^{n}.
\]    
\end{theorem}
In other words, they showed that a typical $r$-dimensional subspace spanned by $r$ vectors in $\{\pm 1\}^{n}$ contains the minimum possible number of $\{\pm 1\}^{n}$ vectors i.e. only the $2r$ vectors consisting of the vectors spanning the subspace and their negatives.

Compared to the setting of Kahn, Koml\'os and Szemer\'edi, our setting has two major differences: 
\begin{enumerate}[(i)]
\item We are interested not in the combinatorial dimension of subspaces spanned by $\{\pm 1\}^{n}$ vectors but of their \emph{orthogonal complements}.
\item The $\{\pm 1\}^{n}$ vectors spanning a subspace in our case are \emph{highly dependent} due to the mutual orthogonality constraint -- indeed, as the proof of the trivial upper bound at the start of the subsection shows, the probability that the rows of a random $k \times n$ $\{\pm1\}$ matrix are mutually orthogonal is $2^{-\Omega(k^{2})}$; this rules out the strategy of conditioning on the rows being orthogonal when $k=\Omega(\sqrt{n})$, even if one were to prove a variant of the result of Kahn, Koml\'os and Szemer\'edi to deal with orthogonal complements.
\end{enumerate}  

Briefly, our approach to dealing with these obstacles is the following. For $k<n$, let $H_{k,n}$ denote a $k\times n$ matrix with all its entries in $\{\pm 1\}$  and all of whose rows are orthogonal.  We will show that there exist absolute constants $0< c_1 < c_2 < 1$ such that if $k\in [c_1 n, c_2 n]$ and if $n$ is sufficiently large, then $H_{k,n}$ must have a certain desirable linear algebraic property; this is the only way in which we use the orthogonality of the rows of $H_{k,n}$, and takes care of (ii). Next, to deal with (i), we will show that for any $k\times n$ matrix $A$ which has this linear algebraic structure, the number of solutions $x$ in $\{\pm 1\}^{n}$ to $Ax=0$ is at most $2^{n-(1+C)k}$, where $C>0$ is a constant depending only on $c_1$ and $c_2$. Using these improved bounds with the same strategy as for the trivial proof, we see that for $n$ sufficiently large, 
\begin{eqnarray*}
H(n) & \leq & \prod_{i=0}^{n-1}2^{n-i}\prod_{i=c_{1}n}^{c_{2}n}2^{-Ck}\\
 & \leq & 2^{{n+1 \choose 2}}2^{-\frac{C(c_{2}^{2}-c_{1}^{2})n^{2}}{2}},
\end{eqnarray*}
which gives the desired improvement. We discuss this in more detail in the next subsection.    

\subsection{Improved Hal\'asz-type inequalities}
As mentioned above, our goal is to study the number of $\{\pm 1\}^{n}$ solutions to an underdetermined system of linear equations $Ax = 0$ possessing some additional structure. This question was studied by Hal\'asz, who proved the following: 

\begin{theorem}[Hal\'asz, \cite{halasz1977estimates}]
\label{thm:halasz-orig}
Let $a_1,\dots, a_n$ be a collection of vectors in $\R^d$. Suppose there exists a constant $\delta > 0$ such that for any unit vector $e\in \R^d$, one can select at least $\delta n$ vectors $a_k$ with $|\langle a_k, e\rangle | \geq 1$. Then, 
\[
\sup_{u\in\R^{d}}\Pr\left[\left\|\sum_{i=1}^{n}\epsilon_{i}a_{i}-u\right\|_{2}< 1\right]\leq c(\delta,d)\left(\frac{1}{\sqrt{n}}\right)^{d},
\]
where $\epsilon_1,\dots,\epsilon_n$ are independent Rademacher random variables i.e. they take the values $\pm 1$ with probability $1/2$ each. 
\end{theorem}

The constant $c(\delta, d)$, which is crucial for our applications, was left implicit by Hal\'asz. However, explicit estimates on this constant may be obtained, as was done by Howard and Oskolkov \cite{howardestimates}.

\begin{theorem}[\cite{howardestimates}]
\label{thm:howard}
Let $a_1,\dots, a_n$ be a collection of vectors in $\R^d$. Suppose that there exists some $m\in\N$ such that for every unit
vector $e\in\R^{d}$, one can select at least $m$ vectors $a_{i_{1}},\dots,a_{i_{m}}$
with $|\langle a_{i_{j}},e\rangle|\geq1/(2\sqrt{d})$ for all $j\in[m]$.
Then, 
\[
\sup_{u\in\R^{d}}\Pr\left[\left\|\sum_{i=1}^{n}\epsilon_{i}a_{i}-u\right\|_{\infty}<\frac{1}{2}\right]\leq C(d)\left(\frac{1}{\sqrt{m}}\right)^{d},
\]
where $C(d)=\left(\frac{\pi^{3/2}d}{\sqrt{2}}\right)^{d}$.
\end{theorem}
\begin{remark}
\label{rmk:tao-vu-exercise}
When $a_1,\dots,a_n$ and $u$ belong to $\Z^{d}$, as will be the case in our applications, the event `$\|\sum_{i=1}^{n}\epsilon_i a_i - u \|_{\infty} < 1/2$' is equivalent to the event `$\sum_{i=1}^{n}\epsilon_i a_i = u$'. In this case, it was noted by Tao and Vu (Exercise 7.2.3 in \cite{tao2006additive}) that the condition $|\langle a_{i_j},e \rangle| \geq 1/(2\sqrt{d})$ may be relaxed to $|\langle a_{i_j},e \rangle| > 0$. However, as stated, their proof still gives a constant $C(d) = \Theta(d)^d$ due to a `duplication' step, which we will show is unnecessary.   
\end{remark}

There are two drawbacks to using the results mentioned above for the kinds of applications we have in mind. Firstly, a constant of the form $C(d) = \Theta(d)^{d}$ does not give any non-trivial information when $d = \Omega(n)$, whereas as discussed in the proof outline, we require an improvement over the Odlyzko bound for $d=\Theta(n)$. Secondly, the hypotheses of these theorems, which involve two quantifiers (`for all' followed by `there exists'), are quite stringent and not easy to verify; in fact, we were unable to find any direct applications of \cref{thm:halasz-orig} in the literature. 

Our key structural observation is that \emph{a `pseudorandom' rectangular matrix contains many disjoint submatrices of large rank}. This motivates replacing the double quantifier hypothesis by a (weaker) hypothesis involving just one existential quantifier which, as we will see, is readily verified to hold in pseudorandom situations. Moreover, while our hypothesis is weaker, we are able to obtain conclusions with asymptotically much better constants, since our structural setting allows us to efficiently leverage the existing rich literature on anti-concentration of sums of independent random variables and anti-concentration of linear images of high dimensional distributions. In particular, we are able to give short and transparent proofs of our inequalities for very general classes of distributions; in contrast, \cref{thm:halasz-orig,thm:howard} hold only for (vector)-weighted sums of independent Rademacher variables, and their proofs involve explicit trigonometric manipulations. We discuss this in more detail in \cref{subsection:anti-concentration,section:proofs-of-inequalities}. 

Our first inequality is a strengthening of \cref{thm:howard} in the setting of \cref{rmk:tao-vu-exercise}, both in terms of the hypothesis and the conclusion. A more general statement appears in \cref{thm:atom-general}. 
\begin{theorem}
\label{thm:halasz-int}
Let $a_{1},\dots,a_{n}$ be a collection vectors in $\R^{d}$ which can be partitioned as $\A_{1},\dots,\A_{\ell}$ with $\ell$ even such that $\dim_{\R^d}(\Span\{a:a\in\A_{i}\})=:r_{i}$.
Then, 
\[
\sup_{u\in \R^{d}}\Pr\left[\sum_{i=1}^{n}\epsilon_{i}a_{i}=u\right]\leq\left(2^{-\ell}{\ell \choose \ell/2}\right)^{\frac{r_{1}+\dots+r_{\ell}}{\ell}}\leq \left(\sqrt{\frac{2}{\pi\ell}}\left(1+O\left(\frac{1}{\ell}\right)\right)\right)^{\frac{r_{1}+\dots+r_{\ell}}{\ell}}.
\]
\end{theorem} 

\begin{remark}
This inequality is tight, as can be easily seen by taking (assuming $n$ is divisible by $d$) $a_{i}$ to be $e_{i\mod d}$,
where $e_{1},\dots,e_{d}$ denotes the standard basis of $\R^{d}$, in which case we can take $\ell = n/d$ and $r_1 = \dots = r_\ell = d$. \end{remark}

To see how \cref{thm:halasz-int} strengthens \cref{thm:howard}, note that the assumptions of \cref{thm:howard} guarantee that there exist $\ell:= \lfloor m/d \rfloor$ disjoint subsets $\A_1,\dots,\A_\ell$ such that $r_1 = \dots = r_\ell = d$. Such a collection of disjoint subsets can be obtained greedily by repeating the following construction $\ell$ times: let $v_1\in \{a_1,\dots,a_n\}$ be any nonzero vector that has not already been chosen in a previous iteration. Having chosen $v_1,\dots,v_s$ for $s < d$, let $u_s \in (\Span\{v_1,\dots,v_s\})^\perp$, and let $v_{s+1}$ be any vector satisfying $|\langle v_{s+1},u_{s}\rangle | > 0$ which has not already been chosen in a previous iteration -- such a vector is guaranteed to exist since there are at least $m$ choices of $v_{s+1}$ by assumption, of which at most $(\ell-1)d < m$ could have been chosen in a previous iteration. It follows that under the assumptions of \cref{thm:howard}, when $a_1,\dots,a_n \in \Z^{d}$, we have:
$$\sup_{u\in\R^{d}}\Pr\left[\left\|\sum_{i=1}^{n}\epsilon_{i}a_{i}-u\right\|_{\infty}<\frac{1}{2}\right]\leq C'(d)\left(\frac{1}{\sqrt{m}}\right)^{d},$$
where $C'(d) \leq \left(\sqrt{\frac{2d}{3}}\right)^{d}$. In particular, we now have a non-trivial bound all the way up to $d=\Theta(m)$, as opposed to just $d = O(\sqrt{m})$ as before. 

Our second inequality is a `small-ball probability' version of \cref{thm:halasz-int}. In order to state it, we need the following definition.  
\begin{definition}
The \emph{stable
rank} of $A$, denoted by $r_{s}(A)$, is defined as 
\[
r_{s}(A):=\left\lfloor\frac{\|A\|_{\text{HS}}^{2}}{\|A\|^{2}}\right\rfloor,
\]
where $\|A\|_{\text{HS}}$ denotes the Hilbert-Schmidt norm of $A$, and $\|A\|$ denotes the operator norm of $A$. 
\end{definition}
\begin{remark}
Recall that $\|A\| = s_1(A)$ and $\|A\|_{\text{HS}}^{2} = \sum_{i=1}^{\text{rank}(A)} s_i(A)^2$, where $s_1(A),s_2(A),\dots$ denote the singular values of $A$ arranged in non-increasing order. Hence, $$r_s(A):=\left\lfloor\frac{\sum_{i=1}^{\text{rank}(A)}s_i(A)^{2}}{s_1(A)^{2}}\right\rfloor;$$
in particular, for any non-zero matrix $A$, $1\leq r_s(A) \leq \text{rank}(A)$, with the right inequality being an equality if and only if $A$ is an orthogonal projection up to an isometry. 
\end{remark}

We can now state our inequality. A more general version appears in \cref{thm:sbp-general}.
\begin{theorem}
\label{thm:halasz-sbp}
Let $a_{1},\dots,a_{n}$ be a collection vectors in $\R^{d}$. For some $\ell \in 2\N$, let $\A_1,\dots, \A_\ell$ be a partition of the set $\{a_1,\dots, a_n\}$, and for each $i \in [\ell]$, let $A_i$ denote the $d\times |\A_i|$ dimensional matrix whose columns are given by the elements of $\A_i$. Then, for every $M \geq 1$ and $\varepsilon \in (0,1)$,
$$\Pr\left[\left\|\sum_{i=1}^{n}\epsilon_ia_i-u\right\|_{2}\leq M\right] \leq 2^{d}\prod_{i=1}^{\ell}\left(\frac{CM}{\sqrt{\varepsilon \ell}\|A_{i}\|_{\text{HS}}}\right)^{\frac{\lceil(1-\varepsilon)r_{s}(A_i)\rceil}{\ell}},$$
where $r_s(A_i)$ denotes the stable rank of $A_i$ and $C$ is an absolute constant.
\end{theorem}
For illustration, consider a situation like above where the set of vectors $a_1,\dots,a_n$ can be partitioned into $m/d$ subsets, each of rank $d$. Assume further that each $a_i$ has norm at least one, so that each of the $m/d$ matrices has Hilbert-Schmidt norm at least $\sqrt{d}$. Then, if the stable rank of each of these matrices is at least $\delta d$ for some $\delta > 0$, it follows that
$$\Pr\left[\left\|\sum_{i=1}^{n}\epsilon_ia_i-u\right\|_{2}\leq 1\right] \leq K^{d},$$
where $K\leq 2\left(C/\sqrt{m}\right)^{\delta/2}$, which is a big improvement over the bound coming from \cref{thm:hadamard} provided that $\delta$ is not too small and $d$ is large. 

\subsection{Counting $\{\pm1\}$-valued normal matrices}
Recall that a matrix $M$ is \emph{normal} if it commutes with its adjoint, i.e., $MM^*=M^*M$ (for real matrices this is the same as $MM^T=M^TM$). Recently, Deneanu and Vu \cite{deneanu2017random} studied the number of $n\times n$ $\{\pm1\}$-valued normal matrices. Since real symmetric matrices are normal, there are at least $2^{\binom{n+1}{2}}$ $\{\pm 1\}$-valued normal matrices. They conjectured that this lower bound is essentially sharp. 
\begin{conjecture}[Deneanu-Vu, \cite{deneanu2017random}]
There are $2^{(0.5+o(1))n^{2}}$ $n\times n$ $\{\pm 1\}$-valued normal matrices. 
\end{conjecture}
As a first non-trivial step towards this conjecture, they showed the following. 
\begin{theorem}[Deneanu-Vu, \cite{deneanu2017random}]
\label{thm:DV}
The number of $n\times n$ $\{\pm 1\}$-valued normal matrices is at most $2^{(c_{DV}+o(1))n^{2}}$ for some constant $c_{DV} < 0.698$. 
\end{theorem}

The problem of counting normal matrices also boils down to the problem of counting the number of solutions to some underdetermined system of linear equations, and using our framework, it is very easy to obtain an upper bound on the number of such matrices of the form $2^{(1-\alpha)n^2}$, for some $\alpha > 0$. Unfortunately, it does not seem that one can get $1-\alpha< c_{DV}$ using this simple method. However, the proof of \cref{thm:DV} in \cite{deneanu2017random} itself uses the Odlyzko bound at a certain stage; therefore, by using their strategy as a black-box, with the application of the Odlyzko bound at this stage replaced by our better bound, we obtain:   
\begin{theorem}
\label{thm:normal}
There exists some $\delta > 0$ such that the number of $n\times n$ $\{\pm 1\}$-valued normal matrices is at most $2^{(c_{DV}-\delta+o(1))n^{2}}$, where $c_{DV}$ denotes the constant in \cite{deneanu2017random}.  
\end{theorem}

\section{Tools}
\subsection{The Fourier transform}
For $p\in [1,\infty)$, let $\L^p(\R^d)$ denote the set of functions $f\colon \R^d \to \C$ such that $\int_{\R^{d}}|f(x)|^{p}dx < \infty$. For $f\in \L^1(\R^d)$, the Fourier transform of $f$ -- denoted by $\widehat{f}$ -- is a function from $\R^{d}$ to $\C$ given by:
\[
\widehat{f}(\xi):=\int_{\R^{d}}f(x)e^{-2\pi i\langle x,\xi\rangle}dx,
\]
where $\langle x,\xi\rangle:=x_{1}\xi_{1}+\dots+x_{d}\xi_{d}$ denotes
the standard inner product on $\R^{d}$. For the reader's convenience, as well as to establish notation, we summarize the following basic properties of the Fourier transform which may be found in any standard textbook on analysis (see, e.g., \cite{rudin2006real}). 
\begin{itemize}
\item (Parseval's formula) Let $f,g \in \L^1(\R^d) \cap \L^2(\R^d)$. Then, the Fourier transforms $\widehat{f}, \widehat{g}$ are also in $\L^2(\R^d)$. Moreover, 
$$\int_{\R^d}f(x)\overline{g(x)}dx = \int_{\R^d}\widehat{f}(\xi)\overline{\widehat{g}(\xi)}d\xi.$$
\item (Convolution formula) For $f,g \in \L^1(\R^d)$, let $f \ast g\colon \R^d \to \C$ denote the convolution of $f$ and $g$ i.e. 
$$ f\ast g(x) = \int_{\R^d}f(x-y)g(y)dy.$$
Then, $f\ast g \in \L^1(\R^d)$, and for any $\xi \in \R^d$ 
$$\widehat{f \ast g}(\xi) = \widehat{f}(\xi)\widehat{g}(\xi).$$
\item (Fourier inversion) Let $f \in \L^1(\R^d)$ be such that $\widehat{f}$ is also in $\L^1(\R^d)$. Then, for any $x \in \R^d$
$$f(x) = \int_{\R^d}\widehat{f}(\xi)e^{2\pi i\langle x,\xi \rangle}d\xi.$$
\item (Fourier transform of autocorrelation) Let $f\in \L^1(\R^d)$ be real-valued, and let $h$ denote the autocorrelation of $f$ i.e. 
$$h(x) := \int_{\R^d}f(y)f(x+y)dy.$$ 
Then, for all $\xi \in \R^d$, 
$$\widehat{h}(\xi) = \left|\widehat{f}(\xi)\right|^{2}.$$ 
\end{itemize}

The notion of Fourier transform extends more generally to finite Borel measures on $\R^d$. For such a measure $\mu$, the Fourier transform is a function from $\R^{d}$ to $\C$ given by:
$$\widehat{\mu}(\xi) := \int_{\R^d}e^{-2\pi i \langle x,\xi \rangle}d\mu(x).$$
To see the connection with the Fourier transform for functions in $\L^1(\R^d)$, note that if the measure $\mu$ is absolutely continuous with respect to the Lebesgue measure $\lambda$, then the density (more precisely, the Radon-Nikodym derivative) $f_\mu:= d\mu/d\lambda$ is in $\L^1(\R^d)$, and we have $\widehat{\mu}(\xi) = \widehat{f_\mu}(\xi)$. 

The only finite Borel measures we will deal with are those which arise as distributions of random vectors valued in $\R^d$. For a $d$-dimensional random vector $X$, let $\mu_X$ denote its distribution. Then, we have (see, e.g., \cite{durrett2010probability}):
\begin{itemize}
\item (Fourier transform of independent random variables) Let $X_1,\dots,X_\ell$ be independent $d$-dimensional random vectors, and let $S_\ell := X_1 + \dots +X_\ell$ denote their sum. Then, for all $\xi \in \R^d$, 
$$\widehat{\mu_{S_\ell}}(\xi) = \prod_{i=1}^{\ell}\widehat{\mu_{X_i}}(\xi).$$ 
\item (Inversion at atoms) Let $X$ be a $d$-dimensional random vector. For any $x \in \R^d$, 
$$\mu_{X}(\{x\}) = \lim_{T_1,\dots,T_d\to \infty}\frac{1}{\text{vol}(B[T_1,\dots,T_d])}\int_{B[T_1,\dots,T_d]}e^{2\pi i \langle x, t\rangle}\widehat{\mu_{X}}(t)dt,$$
where $B[T_1,\dots, T_d]$ denotes the box $[-T_1,T_1]\times \dots \times [-T_d,T_d]$.  
\item (Fourier transform of origin-symmetric random vectors) Let $X$ be a $d$-dimensional, origin-symmetric random vector i.e. $\mu_{X}(x) = \mu_{X}(-x)$ for all $x\in \R^d$. Then, $\widehat{\mu_X}$ is a real-valued function.  
\end{itemize}

\subsection{Anti-concentration}
\label{subsection:anti-concentration}
\begin{definition}
For a random vector $X$ valued in $\R^d$, its (Euclidean) \emph{L\'evy concentration function} $\L(X,\cdot)$ is a function from $\R^{\geq 0}$ to $\R$ defined by:
$$\L(X,\delta) := \sup_{u\in \R^d}\Pr[\|X-u\|_2 \leq \delta].$$
\end{definition}
Anti-concentration inequalities seek to upper bound the L\'evy concentration function for various values of $\delta$. In the discrete setting, a particularly important case is $\delta=0$, which corresponds to the size of the largest atom in the distribution of the random variable $X$.
The proofs of our Hal\'asz-type inequalities will exploit two very general anti-concentration phenomena. 

The first principle states that sums of independent random variables do not concentrate much more than sums of suitable independent Gaussians. In particular, for the weighted sum of independent Rademacher variables, Erd\H{o}s gave a beautiful combinatorial proof to show (improving on a previous bound of Littlewood and Offord) the following.
\begin{theorem}[Erd\H{o}s, \cite{erdos1945lemma}]
\label{ELO}
Let $a=(a_{1},\dots,a_{n})$ be a vector in $\R^{n}$ all of whose
entries are nonzero. Let $S_{a}$ denote the random sum $\epsilon_{1}a_{1}+\dots+\epsilon_{n}a_{n}$, where the $\epsilon_{i}$'s are independent Rademacher random variables. Then, 
\[
\sup_{c\in\R}\Pr[S_{a}=c]\leq\frac{{n \choose \lfloor n/2\rfloor}}{2^{n}} \sim \sqrt{\frac{2}{\pi n}}.
\]
\end{theorem}

Up to a constant, this was subsequently generalized by Rogozin to handle the L\'evy concentration function of sums of general independent random variables. 
\begin{theorem}[Rogozin, \cite{rogozin1961increase}]
\label{Rogozin's-thm}
There exists a universal constant $C>0$ such that for any independent
random variables $X_{1},\dots,X_{n}$, and any $r>0$, we have 
\[
\L({S_n},\delta)\leq\frac{C}{\sqrt{\sum_{i=1}^{n}\left(1-\L({X_{i}},\delta)\right)}},
\]
where $S_n:= X_1+\dots +X_n$.
\end{theorem}

The second anti-concentration principle concerns random vectors of the form $AX$, where $A$ is a fixed $m\times n$ matrix, and $X=(X_1,\dots,X_n)$ is a random vector with independent coordinates. It states roughly that if the $X_i$'s are anti-concentrated on the line, and if $A$ has large rank in a suitable sense, then the random vector $AX$ is anti-concentrated in space \cite{rudelson2014small}. 

As a first illustration of this principle, we present the following lemma, which may be viewed as a `tensorization' of the Erd\H{o}s-Littlewood-Offord inequality. 
\begin{lemma}
\label{lemma:ELO-large-rank}
Let $A$ be an $m\times n$ matrix (where $m\leq n$) of rank $r$,
and let $X$ be a random vector distributed uniformly on $\{\pm1\}^{n}$.
Then for any $\ell\in\N$, 
\[
\sup_{u\in \R^{m}}\Pr[AX^{(1)}+\dots+AX^{(\ell)}=u]\leq\left(2^{-\ell}{\ell \choose \lfloor \ell/2\rfloor}\right)^{r},
\]
where $X^{(1)},\dots,X^{(\ell)}$ are i.i.d. copies of $X$. 
\end{lemma}

\begin{proof}
By relabeling the coordinates if needed, we may write $A$ as a
block matrix $\left(\begin{array}{cc}
E & F\\
G & H
\end{array}\right)$where $E$ is an $r\times r$ invertible matrix, $F$ is an $r\times(n-r)$
matrix, $G$ is a $(m-r)\times r$ matrix, and $H$ is a $(m-r)\times(n-r)$
matrix. Let $B$ denote the invertible $m\times m$ matrix $\left(\begin{array}{cc}
E^{-1} & 0\\
0 & I_{m-r}
\end{array}\right)$, and note that $BA=\left(\begin{array}{cc}
I_{r} & \ast\\
\ast & \ast
\end{array}\right)$. For a vector $v\in\R^{s}$ with $s\geq r$, let $Q_{r}(v)\in\R^{r}$
denote the vector consisting of the first $r$ coordinates of $v$.
Also, let $X_{i}^{(j)}$ denote the $i^{th}$ coordinate of the random
vector $X^{(j)}$, let $\mathcal{R}$ denote the collection of
random variables $\{X_{r+1}^{(1)},\dots,X_{n}^{(1)},\dots X_{r+1}^{(\ell)},\dots,X_{n}^{(\ell)}\}$, and let $\mathcal{S}$ denote the collection of random variables $\{X^{(1)}_{1},\dots,X^{(1)}_{r},\dots,X^{(\ell)}_{1},\dots,X^{(\ell)}_{r}\}$.
Then for any $u \in \R^m$, we have:
\begin{eqnarray*}
\Pr\left[AX^{(1)}+\dots+AX^{(\ell)}=u\right] & = & \Pr\left[BAX^{(1)}+\dots+BAX^{(\ell)}=Bu\right]\\
 & \leq & \Pr\left[Q_{r}(BAX^{(1)})+\dots+Q_{r}(BAX^{(\ell)})=Q_{r}(Bu)\right]\\
 & = & \E_{\mathcal{R}}\left[\Pr_{\mathcal{S}}\left[Q_{r}(BAX^{(1)})+\dots+Q_{r}(BAX^{(\ell)})=Q_r(Bu)|\mathcal{R}\right]\right]\\
 & = & \E_{\mathcal{R}}\left[\Pr_{\mathcal{S}}\left[Q_{r}(X^{(1)})+\dots+Q_{r}(X^{(\ell)})=f(\mathcal{R})|\mathcal{R}\right]\right]\\
 & = & \E_{\mathcal{R}}\left[\prod_{i=1}^{r}\Pr\left[X_{i}^{(1)}+\dots+X_{i}^{(\ell)}=f_{i}(\mathcal{R})|\mathcal{R}\right]\right]\\
 & \leq & \E_{\mathcal{R}}\left[\prod_{i=1}^{r}2^{-\ell}{\ell \choose \ell/2}\right]\\
 & = & \left(2^{-\ell}{\ell \choose \ell/2}\right)^{r},
\end{eqnarray*}
where the third line follows from the law of total probability; the
fourth line follows from the explicit form of $BA$ mentioned above;
the fifth line follows from the independence of the coordinates of
$X^{(j)}$; and the sixth line follows from the Erd\H{o}s-Littlewood-Offord inequality (\cref{ELO}). Taking the supremum over $u \in \R^{m}$ completes the proof. 
\end{proof}

\begin{remark}
\label{rmk:elo-rogozin-large-rank}
By using Rogozin's inequality (\cref{Rogozin's-thm}) instead of the Erd\H{o}s-Littlewood-Offord inequality, we may generalize the lemma to handle any random vector $X=(X_1,\dots,X_n)$ with independent coordinates $X_i$, provided we replace the conclusion by 
\[
\sup_{u\in \R^{m}}\Pr[AX^{(1)}+\dots+AX^{(\ell)}=u]\leq \left(\frac{C}{\ell}\right)^{r/2}\times \max_{I\subseteq[n], |I|=r}\prod_{i\in I}\frac{1}{\sqrt{1-\L({X_i},0)}},
\]
where $C$ is a universal constant. 
\end{remark}

For the L\'evy concentration function for general $\delta$, a version of \cref{lemma:ELO-large-rank} was proved by Rudelson and Vershynin in \cite{rudelson2014small}. 

\begin{theorem}[Rudelson-Vershynin, \cite{rudelson2014small}]
\label{thm:RV}
Consider a random vector $X=(X_{1},\dots,X_{d})$ where $X_{i}$ are
real-valued independent random variables. Let $\delta,\rho\geq0$ be such that for all $i\in[d]$,
\[
\L(X_{i},\delta)\leq\rho.
\]
Then, for every $m\times n$ matrix $A$, every $M\geq1$ and every
$\varepsilon\in(0,1)$, we have 
\[
\L\left(AX,M\delta\|A\|_{\text{HS}}\right)\leq\left(C_{\varepsilon}M\rho\right)^{\lceil(1-\varepsilon)r_{s}(A)\rceil},
\]
where $C_{\varepsilon}=C/\sqrt{\varepsilon}$ for some absolute constant
$C > 0$. 
\end{theorem}
More general statements of a similar nature may be found in \cite{rudelson2014small}. 

\subsection{The replication trick}
\label{subsection:replication}
In this section, we present the `replication trick', which allows us to reduce considerations about anti-concentration of sums of independent random vectors to considerations about anti-concentration of sums of independent $\emph{identically distributed}$ random vectors. This will be useful since the `correct' analog of Rogozin's inequality for general random vectors with independently coordinates is not available; to our knowledge, the best result in this direction is due to Esseen \cite{esseen1966kolmogorov}, who proved an inequality of this form for such random vectors satisfying  additional symmetry conditions, which will not be available in our applications. The statement/proof of the `atomic' version of the replication trick (\cref{prop:replication-atom}) is similar in spirit to Corollaries 7.12 and 7.13 in \cite{tao2006additive} with an important difference: we have no need for the lossy `domination' and `duplication' steps in \cite{tao2006additive}; instead, we ensure the non-negativity of the Fourier transform at various places by using the previously stated simple fact that the Fourier transform of the distribution of an origin-symmetric random vector is real valued, and restricting ourselves to even powers thereof. 
\begin{proposition}
\label{prop:replication-atom}
Let $X_{1},\dots,X_{n}$ be independent random vectors valued in $\R^{d}$.
For each $i\in[n]$, let $\tilde{X}_{i}:=X_{i}-X_{i}'$, where $X_{i}'$
is an independent copy of $X_{i}$. Let $S_{n}:=X_{1}+\dots+X_{n}$,
and for any $i\in[n]$, $m\in\N$, let $\tilde{S}_{i,m}:=\tilde{X}_{i}^{(1)}+\dots\tilde{X}_{i}^{(m)}$,
where $\tilde{X}_{i}^{(1)},\dots,\tilde{X}_{i}^{(m)}$ are independent
copies of $\tilde{X}_{i}$. Then for any $v\in\R^{d}$,
\[
\Pr\left[S_{n}=v\right]\leq\prod_{i=1}^{n}\Pr\left[\tilde{S}_{i,a_{i}/2}=0\right]^{\frac{1}{a_{i}}}
\]
for any $a_{1},\dots,a_{n}\in4\cdot\N$ 
such that $a_{1}^{-1}+\dots+a_{n}^{-1}=1$.
\end{proposition}
Here, $4\cdot \N$ denotes the subset of natural numbers given by $\{4m \colon m \in \N\}$.
\begin{proof}
As before, we let $\mu_X$ denote the distribution of the $d$-dimensional random vector $X$. We have:
\begin{eqnarray*}
\mu_{S_{n}}(v) & = & \lim_{T_{1},\dots,T_{d}\to\infty}\frac{1}{\text{vol}B[T_{1},\dots,T_{d}]}\int_{B[T_{1},\dots,T_{d}]}e^{-i\langle t,v\rangle}\widehat{\mu_{S_{n}}}(t)dt\\
 & = & \lim_{T_{1},\dots,T_{d}\to\infty}\frac{1}{\text{vol}B[T_{1},\dots,T_{d}]}\int_{B[T_{1},\dots,T_{d}]}e^{-i\langle t,v\rangle}\prod_{i=1}^{n}\widehat{\mu_{X_{i}}}(t)dt\\
 & \leq & \lim_{T_{1},\dots,T_{d}\to\infty}\frac{1}{\text{vol}B[T_{1},\dots,T_{d}]}\prod_{i=1}^{n}\left(\int_{B[T_{1},\dots,T_{d}]}\left|\widehat{\mu_{X_{i}}}(t)\right|^{a_{i}}dt\right)^{\frac{1}{a_{i}}}\\
 & = & \lim_{T_{1},\dots,T_{d}\to\infty}\frac{1}{\text{vol}B[T_{1},\dots,T_{d}]}\prod_{i=1}^{n}\left(\int_{B[T_{1},\dots,T_{d}]}\left(\widehat{\mu_{\tilde{X_{i}}}}(t)\right)^{\frac{a_{i}}{2}}dt\right)^{\frac{1}{a_{i}}}\\
 & = & \lim_{T_{1},\dots,T_{d}\to\infty}\frac{1}{\text{vol}B[T_{1},\dots,T_{d}]}\prod_{i=1}^{n}\left(\int_{B[T_{1},\dots,T_{d}]}\widehat{\mu_{\tilde{S}_{i,a_{i}/2}}}(t)dt\right)^{\frac{1}{a_{i}}}\\
 & = & \prod_{i=1}^{n}\left(\lim_{T_{1},\dots,T_{d}\to\infty}\frac{1}{\text{vol}B[T_{1},\dots,T_{d}]}\int_{B[T_{1},\dots,T_{d}]}\widehat{\mu_{\tilde{S}_{i,a_{i}/2}}}(t)dt\right)^{\frac{1}{a_{i}}}\\
 & = & \prod_{i=1}^{n}\left(\mu_{\tilde{S}_{i,a_{i}/2}}(0)\right)^{\frac{1}{a_{i}}},
\end{eqnarray*}
where the first line follows from the Fourier inversion formula at atoms; the second line follows from the independence of $X_1,\dots,X_n$; the third line follows from H\"older's inequality; the fourth line follows from the fact that $\widehat{\mu_{\tilde{X}_i}}(t) = |\widehat{\mu_{X_i}}(t)|^{2}$ (since the distribution of $\tilde{X_i}$ is the autocorrelation of the distribution of $X_i$); the fifth line follows from the independence of $\tilde{X}_i^{(1)}\dots,\tilde{X}_i^{(a_i/2)}$; and the last line follows again from the Fourier inversion formula at atoms. 
\end{proof}
\begin{remark}
\label{rmk:replication-atom}
The same proof shows that when $X_{1},\dots,X_{n}$ are independent
\emph{origin symmetric }random vectors, then for any $v\in\R^{d}$
\[
\Pr\left[S_{n}=v\right]\leq\prod_{i=1}^{n}\Pr\left[S_{i,a_{i}}=0\right]^{\frac{1}{a_{i}}}
\]
for any $a_{1},\dots,a_{n}\in2\N$ such that $a_{1}^{-1}+\dots+a_{n}^{-1}=1$, where $S_{i,a_i}$ denotes the sum of $a_i$ independent copies of $X_i$.  
\end{remark}
The next proposition is a version of \cref{prop:replication-atom} for the L\'evy concentration function. Essentially the same proof can also be used to prove variants for norms other than the Euclidean norm. 
\begin{proposition}
\label{prop:replication-sbp}
Let $X_{1},\dots,X_{n}$ be independent random vectors valued in $\R^{d}$.
For each $i\in[n]$, let $\tilde{X}_{i}:=X_{i}-X_{i}'$, where $X_{i}'$
is an independent copy of $X_{i}$. Let $S_{n}:=X_{1}+\dots+X_{n}$,
and for any $i\in[n]$, $m\in\N$, let $\tilde{S}_{i,m}:=\tilde{X}_{i}^{(1)}+\dots\tilde{X}_{i}^{(m)}$,
where $\tilde{X}_{i}^{(1)},\dots,\tilde{X}_{i}^{(m)}$ are independent
copies of $\tilde{X}_{i}$. Then for any $\delta > 0$,
\[
\L(S_n,\delta)\leq2^{d}\prod_{i=1}^{n}\L(\tilde{S}_{i,a_i/2},4\delta)^{1/a_i}
\]
for any $a_{1},\dots,a_{n}\in4\N$ such that $a_{1}^{-1}+\dots+a_{n}^{-1}=1$.
\end{proposition}

\begin{proof}
Let $\boldsymbol{1}_{B_{\delta}(0)}$ denote the indicator function of
the ball of radius $\delta$ centered at the origin. We will make use of
the readily verified elementary inequality 
\begin{equation}
\label{eqn:convolution-switch}
\text{vol}(B_{\delta}(0))\boldsymbol{1}_{B_{\delta}(0)}(x)\leq\boldsymbol{1}_{B_{2\delta}(0)}\ast\boldsymbol{1}_{B_{2\delta}(0)}(x)\leq\text{vol}(B_{2\delta}(0))\boldsymbol{1}_{B_{4\delta}(0)}(x).
\end{equation}
By adding to each $X_i$ an independent random vector with distribution given by a `bump function' with arbitrarily small support around the origin, we may assume that the distributions of all the random vectors under consideration are absolutely continuous with respect to the Lebesgue measure on $\R^{d}$, and thus have densities. For such a random vector $Y$, we will denote its density with respect to the $d$-dimensional Lebesgue measure by $f_Y$. 
Then, for any $v\in \R^d$, we have:
\begin{eqnarray*}
\Pr\left[\|S_{n}-v\|_{2}\leq \delta\right] & = & \int_{x\in\R^{d}}\boldsymbol{1}_{B_{\delta}(0)}(x)f_{S_{n}}(x+v)dx\\
 & \leq & \text{vol}(B_{\delta}(0))^{-1}\int_{x\in\R^{d}}\left(\boldsymbol{1}_{B(2\delta)}\ast\boldsymbol{1}_{B(2\delta)}\right)(x)f_{S_{n}}(x+v)dx\\
 & = & \text{vol}(B_{\delta}(0))^{-1}\int_{\xi\in\R^{d}}e^{2\pi i\langle\xi,v\rangle}\left(\boldsymbol{1}_{B(2\delta)}\ast\boldsymbol{1}_{B(2\delta)}\right)^{\wedge}(\xi)\widehat{f_{S_{n}}}(\xi)d\xi\\
 & = & \text{vol}(B_{\delta}(0))^{-1}\int_{\xi\in\R^{d}}e^{2\pi i\langle\xi,v\rangle}\left(\widehat{\boldsymbol{1}_{B(2\delta)}}(\xi)\right)^{2}\prod_{i=1}^{n}\widehat{f_{X_{i}}}(\xi)d\xi\\
 & = & \text{vol}(B_{\delta}(0))^{-1}\int_{\xi\in\R^{d}}e^{2\pi i\langle\xi,v\rangle}\prod_{i=1}^{n}\left(\left(\widehat{\boldsymbol{1}_{B(2\delta)}}(\xi)\right)^{\frac{2}{a_{i}}}\widehat{f_{X_{i}}}(\xi)\right)d\xi\\
 & \leq & \text{vol}(B_{\delta}(0))^{-1}\prod_{i=1}^{n}\left(\int_{\xi\in\R^{d}}\left(\widehat{\boldsymbol{1}_{B(2\delta)}}(\xi)\right)^{2}\left|\widehat{f_{X_{i}}}(\xi)\right|^{a_{i}}d\xi\right)^{\frac{1}{a_{i}}}\\
 & = & \text{vol}(B_{\delta}(0))^{-1}\prod_{i=1}^{n}\left(\int_{\xi\in\R^{d}}\left(\widehat{\boldsymbol{1}_{B(2\delta)}}(\xi)\right)^{2}\left(\widehat{f_{\tilde{X}_{i}}}(\xi)\right)^{\frac{a_{i}}{2}}d\xi\right)^{\frac{1}{a_{i}}}\\
 & = & \text{vol}(B_{\delta}(0))^{-1}\prod_{i=1}^{n}\left(\int_{\xi\in\R^{d}}\left(\boldsymbol{1}_{B(2\delta)}\ast\boldsymbol{1}_{B(2\delta)}\right)^{\wedge}(\xi)\widehat{f_{\tilde{S}_{i,a_{i}/2}}}(\xi)d\xi\right)^{\frac{1}{a_{i}}}\\
 & = & \text{vol}(B_{\delta}(0))^{-1}\prod_{i=1}^{n}\left(\int_{x\in\R^{d}}\left(\boldsymbol{1}_{B(2\delta)}\ast\boldsymbol{1}_{B(2\delta)}\right)(x)f_{\tilde{S}_{i,a_{i/2}}}(x)dx\right)^{\frac{1}{a_{i}}}\\
 & \leq & \text{vol}(B_{\delta}(0))^{-1}\text{vol}(B_{2\delta}(0))\prod_{i=1}^{n}\left(\int_{x\in\R^{d}}\boldsymbol{1}_{B(4\delta)}(x)f_{\tilde{S}_{i,a_{i/2}}}(x)dx\right)^{\frac{1}{a_{i}}}\\
 & = & 2^{d}\prod_{i=1}^{n}\left(\Pr\left[\|\tilde{S}_{i,a_{i/2}}\|_{2}\leq4\delta\right]\right)^{\frac{1}{a_{i}}},
\end{eqnarray*}
where the second line follows from \cref{eqn:convolution-switch}; the third line follows from Parseval's formula; the fourth line follows from the convolution formula and the independence of $X_1,\dots,X_n$; the sixth line follows from H\"older's inequality, along with the fact that $\widehat{\textbf{1}_{B(2\delta)}}(\xi)$ is real valued for all $\xi \in \R^{d}$; the seventh line follows from the fact that $|\widehat{f_{X_i}}(\xi)|^{2} = \widehat{f_{\tilde{X}_i}}(\xi)$ for all $\xi \in \R^{d}$; the ninth line follows again from Parseval's formula; and the tenth line follows from \cref{eqn:convolution-switch}. 
Taking the supremum over all $v\in \R^{d}$ gives the desired conclusion. 
\end{proof}
\begin{remark}
\label{rmk:replication-sbp}
As in \cref{rmk:replication-atom}, if $X_1,\dots,X_n$ are origin-symmetric, then the same conclusion holds with $\tilde{S}_{i,a_i/2}$ replaced by $S_{i,a_i}$, for any $a_1,\dots,a_n \in 2\N$ with $a_1^{-1}+\dots+a_n^{-1} = 1$.
\end{remark}

\section{Proofs}
\subsection{Proofs of Hal\'asz-type inequalities}
\label{section:proofs-of-inequalities}
By combining the tools from \cref{subsection:anti-concentration,subsection:replication}, we can now prove our Hal\'asz-type inequalities. All of them follow the same general outline. We begin by proving \cref{thm:halasz-int}. 
\begin{proof}[Proof of \cref{thm:halasz-int}]
Let $\A_{1},\dots,\A_{\ell}$ be the partition of $\{a_{1},\dots,a_{n}\}$
as in the statement of the theorem. For each $i\in[\ell]$, let $A_{i}$
denote $d\times|\A_{i}|$ dimensional matrix whose columns are given
by the elements of $\A_{i}$. With this notation, we can rewrite the
random vector $\sum_{i=1}^{n}\epsilon_{i}a_{i}$ as $\sum_{j=1}^{\ell}A_{j}Y_{j}$,
where $Y_{j}$ is uniformly distributed on $\{\pm1\}^{|\A_{j}|}$
and $Y_{1},\dots,Y_{\ell}$ are independent.

Since the random vectors $X_{1}:=A_{1}Y_{1},\dots,X_{n}:=A_{n}Y_{n}$
are origin-symmetric, and since $\ell\in2\N$, it follows from \cref{prop:replication-atom,rmk:replication-atom} that for any $u\in\R^{d}$, 
\begin{eqnarray*}
\Pr\left[\sum_{i=1}^{n}\epsilon_{i}a_{i}=u\right] & = & \Pr\left[\sum_{j=1}^{\ell}X_{j}=u\right]\\
 & \leq & \prod_{i=1}^{\ell}\Pr\left[X_{j}^{(1)}+\dots+X_{j}^{(\ell)}=0\right]^{\frac{1}{\ell}},
\end{eqnarray*}
where $X_j^{(1)},\dots, X_j^{(\ell)}$ are i.i.d. copies of $X_j$. Further, since $\text{rank}(A_{j})=r_{j}$ by assumption, it follows
from \cref{lemma:ELO-large-rank} that 
\begin{eqnarray*}
\Pr\left[X_{j}^{(1)}+\dots+X_{j}^{(\ell)}=0\right] & = & \Pr\left[A_{j}Y_{j}^{(1)}+\dots+A_{j}Y_{j}^{(\ell)}=0\right]\\
 & \leq & \left(2^{-\ell}{\ell \choose \ell/2}\right)^{r_{j}}.
\end{eqnarray*}
Substituting this bound in the previous inequality completes the proof. 
\end{proof}

By using \cref{rmk:elo-rogozin-large-rank} instead of \cref{lemma:ELO-large-rank}, we can use the same proof to obtain the following more general statement. 
\begin{theorem}
\label{thm:atom-general}
Let $a_{1},\dots,a_{n}$ be a collection vectors in $\R^{d}$ which can be partitioned as $\A_{1},\dots,\A_{\ell}$ such that $\dim_{\R^d}(\Span\{a:a\in\A_{i}\})=:r_{i}$. Let $x_1,\dots,x_n$ be independent random variables, and for each $i\in [n]$, let $\tilde{x}_i:= x_i - x_i'$, where $x_i'$ is an independent copy of $x_i$. Then, 
\[
\sup_{u\in \R^{d}}\Pr\left[\sum_{i=1}^{n}x_{i}a_{i}=u\right]\leq 2^{d}\inf_{(b_1,\dots,b_\ell)\in \mathcal{B}}\left(\frac{C}{\ell\lambda}\right)^{\sum_{i=1}^{\ell}\frac{r_i}{2b_i}}, 
\]
where $\lambda := \min_{i\in [n]}(1-\L_{\tilde{x}_i}(0)) $ and $\mathcal{B} = \{(b_1,\dots,b_\ell)\in (4\N)^\ell : b_1^{-1}+\dots + b_\ell^{-1} = 1\}$.
\end{theorem}

We now state and prove the general small-ball version of our anti-concentration inequality. 
\begin{theorem}
\label{thm:sbp-general}
Let $a_{1},\dots,a_{n}$ be a collection vectors in $\R^{d}$. Let $\A_1,\dots, \A_\ell$ be a partition of the set $\{a_1,\dots, a_n\}$, and for each $i \in [\ell]$, let $A_i$ denote the $d\times |\A_i|$ dimensional matrix whose columns are given by the elements of $\A_i$. Let $x_1,\dots,x_n$ be independent random variables, and for each $i\in [n]$, let $\tilde{x}_i:= x_i - x_i'$, where $x_i'$ is an independent copy of $x_i$. Let $\delta,\lambda \geq 0$ be such that $\min_{i\in[n]}(1-\L(\tilde{x}_i,\delta)) = \lambda$. Then, for every $M \geq 1$ and $\varepsilon \in (0,1)$,
$$\L\left(\sum_{i=1}^{n}x_ia_i, M\delta\right) \leq \inf_{(b_1,\dots,b_\ell)\in\mathcal{B}}\prod_{i=1}^{\ell}\left(\frac{CM}{\sqrt{\varepsilon b_i\lambda}\|A_{i}\|_{\text{HS}}}\right)^{\frac{\lceil(1-\varepsilon)r_{s}(A_i)\rceil}{b_i}},$$
where $r_s(A_i)$ denotes the stable rank of $A_i$, $C$ is an absolute constant, and $\mathcal{B} = \{(b_1,\dots,b_\ell)\in (4\N)^\ell : b_1^{-1}+\dots + b_\ell^{-1} = 1\}$.
\end{theorem}
\begin{proof}
As before, we begin by rewriting the random vector $\sum_{i=1}^{n}x_{i}a_{i}$
as $\sum_{i=1}^{\ell}A_{i}Y_{i}$. From \cref{prop:replication-sbp}, it follows that for any $(b_{1},\dots,b_{\ell})\in\mathcal{B}$,
\begin{eqnarray*}
\L\left(\sum_{i=1}^{n}x_{i}a_{i},M\delta\right) & = & \L\left(\sum_{i=1}^{\ell}A_{i}Y_{i},M\delta\right)\\
 & \leq & 2^{d}\prod_{i=1}^{\ell}\L\left(A_{i}\left(\tilde{Y}_{i}^{(1)}+\dots+\tilde{Y}_{i}^{(b_{i}/2)}\right),4M\delta\right)^{\frac{1}{b_{i}}}.
\end{eqnarray*}
Next, since $1-\L(\tilde{x}_{i},\delta)\geq\lambda$ for all $i\in[n]$, it follows from \cref{Rogozin's-thm} that $$\L\left(\tilde{x}_i^{1}+\cdots+\tilde{x}_i^{(b_i/2)},\delta\right) \leq \frac{C}{\sqrt{b_i\lambda}},$$
where $C$ is an absolute constant. In particular, all of the (independent) coordinates of the random vector $\tilde{Y}_{i}^{(1)}+\dots+\tilde{Y}_{i}^{(b_{i}/2)}$ have $\delta$-L\'evy concentration function bounded by $C/\sqrt{b_i\lambda}$. Hence, it follows from \cref{thm:RV} that
\[
\L\left(A_{i}\left(\tilde{Y}_{i}^{(1)}+\dots+\tilde{Y}_{i}^{(b_{i}/2)}\right),4M\delta\right)\leq\left(\frac{CM}{\sqrt{\varepsilon b_i\lambda}\|A_{i}\|_{\text{HS}}}\right)^{\lceil(1-\varepsilon)r_{s}(A_i)\rceil},
\]
where $C$ is an absolute constant. Substituting this in the first inequality completes the proof.  
\end{proof}
\begin{remark}
When the $x_i$'s are origin symmetric random variables, we may use \cref{rmk:replication-sbp} instead of \cref{prop:replication-sbp} to obtain a similar conclusion -- with the infimum now over the larger set $\mathcal{B}'=\{(b_1,\dots,b_\ell)\in (2\N)^{\ell}\colon b_1^{-1}+\dots+b_\ell^{-1}=1\}$ -- under the assumption that $\min_{i\in[n]}(1-\L(x_i,\delta))=\lambda$. In particular, if $\ell$ is even, then taking $b_1=\dots=b_\ell = \ell$ gives \cref{thm:halasz-sbp}.  
\end{remark}

\subsection{Proof of \cref{thm:hadamard}}
As in \cref{subsec-approach}, let $H_{k,n}$ denote a $k\times n$ matrix with all its entries in $\{\pm 1\}$ and all of whose rows are orthogonal. For convenience of notation, we isolate the following notion. 
\begin{definition}\label{def: rank partition}
For any $r,\ell \in \N$, a matrix $M$ is said to admit an $(r,\ell)$-\emph{rank partition} if there exists a decomposition of the columns of $M$ into $\ell$ disjoint subsets, each of which corresponds to a submatrix of rank at least $r$.
\end{definition}
Note that the existence of an $(r,\ell)$-rank partition is a uniform version of the condition appearing in \cref{thm:halasz-int}. The next proposition shows that any $H_{k,n}$ with $k$ admits an $(r,\ell)$-rank partition with $r$ and $\ell$ sufficiently large. 
\begin{proposition}
\label{prop:rank-partition-hadamard}
Let $r,\ell \in \N$ such that $2\leq \ell, r \leq k$ and $(e^{2}\ell)^{k}<(n/r)^{k-r}$. Then, $H_{k,n}$ admits an $(r,\ell)$-rank partition.
\end{proposition}
\begin{proof}
The proof proceeds in two steps -- first, we show that $H_{k,n}$ contains many non-zero $k\times k$ minors, and second, we apply a simple greedy procedure to these non-zero minors to produce an $(r,\ell)$-rank partition for the desired values of $r$ and $\ell$. 

The first step follows easily from the classical Cauchy-Binet formula (see, e.g., \cite{aigner2010proofs}), which asserts that:
$$\det(H_{k,n}H_{k,n}^T) = \sum_{A \in \mathcal{M}_k}\det(A)^{2},$$
where $\mathcal{M}_k$ denotes the set of all $k\times k$ submatrices of $H_{k,n}$. In our case, $H_{k,n}H_{k,n}^{T} = n\text{Id}_{k}$, so that $\det(H_{k,n}H_{k,n}^{T}) = n^{k}$. Moreover, since each $A \in \mathcal{M}_k$ is a $k\times k$ $\{\pm 1\}$-valued matrix, $\det(A)^{2} \leq k^{k}$ (with equality attained if and only if $A$ is itself a Hadamard matrix). Hence, it follows from the Cauchy-Binet formula that $H_{n,k}$ has at least $(n/k)^{k}$ non-zero minors. 

Next, we use these non-zero minors to construct an $(r,\ell)$-rank partition in $\ell$ steps as follows: In Step $1$, choose $r$ columns of an arbitrary non-zero minor -- such a minor is guaranteed to exist by the discussion above. Let $\mathcal{C}_k$ denote the union of the columns chosen by the end of Step $k$, for any $1\leq k \leq \ell -1$. In Step $k+1$, we choose $r$ linearly independent columns which are disjoint from $\mathcal{C}_k$. Then, the $\ell$ collections of $r$ columns chosen at different steps gives an $(r,\ell)$-rank partition of $H_{k,n}$. 

Therefore, to complete the proof, it only remains to show that for each $1\leq k\leq \ell-1$, there is a choice of $r$ linearly independent columns which are disjoint from $\mathcal{C}_k$. Since $|\mathcal{C}_k|=rk$, this is in turn implied by the stronger statement that there is a choice of $r$ linearly independent columns which are disjoint from any collection $\mathcal{C}$ of at most $r\ell$ columns. In order to see this, we note that the number of $k\times k$ submatrices of $H_{k,n}$ which have at least $k-r$ columns contained in $\mathcal{C}$ is at most:
\begin{eqnarray*}
\sum_{s=0}^{r}{r\ell \choose k-s}{n \choose s} & \leq & {r\ell \choose k}\sum_{s=0}^{r}{n \choose s} 
\\
 & \leq & \left(\frac{er\ell}{k}\right)^{k}\left(\frac{en}{r}\right)^{r}\\ 
 & < & \left(\frac{n}{k}\right)^{k},
\end{eqnarray*}
where the first inequality uses $2\leq \ell$ and the final inequality follows by assumption. Since there are at least $(n/k)^{k}$ non-zero minors of $H_{k,n}$, it follows that there exists a $k\times k$ submatrix $A_{k+1}$ of $H_{n,k}$ of full rank which shares at most $k-r$ columns with $\mathcal{C}_k$. In particular, $A_{k+1}$ contains $r$ linearly independent columns which are disjoint from $\mathcal{C}_k$, as desired. 
\end{proof}

The previous proposition essentially completes the proof of \cref{thm:hadamard}. Indeed, recall from \cref{subsec-approach} that it suffices to show the following: there exist absolute constants $0<c_1<c_2 <1$ and $C>0$ such that for all $k\in [c_1n,c_2n]$, the number of solutions $x\in \{\pm 1\}^{n}$ to $H_{k,n}x = 0$ is at most $2^{-(1+C)k}$. The previous proposition shows that $H_{k,n}$ admits an $(r,\ell)$-rank partition with $r=\lfloor k/2\rfloor$  and $\ell = \lfloor\sqrt{n/ke^{4}}\rfloor$. Hence, from \cref{thm:halasz-int}, it follows that for $k\in [1,n/15000]$, the number of solutions $x\in \{\pm 1\}^{n}$ to $H_{k,n}x = 0$ is at most $2^{n-(1+1/10)k}$,
which completes the proof. 
\begin{remark}
For our problem of providing an upper bound on the number of Hadamard matrices, we could have used the somewhat simpler \cref{lemma:large rank for many} (instead of \cref{prop:rank-partition-hadamard}), which shows that there are very few $H_{k,n}$ which do not admit an $(r,\ell)$-rank partition for sufficiently large $r,\ell$. However, we used \cref{prop:rank-partition-hadamard} to show that it is easy to find such a rank partition even for a given $k\times n$ system of linear equations $A$ -- indeed, the proof of \cref{prop:rank-partition-hadamard} goes through as long as $\det(AA^T)$ is `large' (which is indeed the case for random or `pseudorandom' $A$), and all $k\times k$ minors of $A$ are uniformly bounded (which is guaranteed in settings where $A$ has restricted entries, as in our case).  
\end{remark}

\subsection{Proof of \cref{thm:normal}}
\label{subsec-proof-normality}
In this section, we show how to obtain a non-trivial upper bound on the number of $\{\pm 1\}$-valued normal matrices using our general framework. As mentioned in the introduction, this bound by itself is not stronger than the one obtained by Deneanu and Vu \cite{deneanu2017random}; however, it can be used in their proof in a modular fashion to obtain an improvement over their bound, thereby proving \cref{thm:normal}. As the proof of Deneanu and Vu is quite  technical, we defer the details of this second step to \cref{appendix-completing-normal-proof}.

Following Deneanu and Vu, we consider the following generalization of the notion of normality: 
\begin{definition}
Let $N$ be a fixed (but otherwise arbitrary) $n\times n$ matrix. An $n\times n$ matrix $M$ is said to be $N$-\emph{normal} if and only if $$MM^T-M^TM=N.$$ 
\end{definition} 
For any $n\times n$ matrix $N$, we let $\mathcal N(N)$ denote the set of all $n\times n$, $\{\pm 1\}$-valued matrices which are $N$-normal. In particular, $\mathcal{N}(0)$ is the set of all $n\times n$, $\{\pm 1\}$-valued normal matrices. The notion of $N$-normality is crucial to the proof of Deneanu and Vu, which is based on an inductive argument -- they show that the quantity $2^{(c_{DV}+o(1))n^{2}}$ in \cref{thm:DV} is actually a \emph{uniform} upper bound on the size of the set $\mathcal{N}(N)$ for any 
$N$. While this general notion of normality is not required to obtain \emph{some} non-trivial upper bound on the number of normal matrices, either using our framework or theirs, we will state and prove the results of this section for $N$-normality, since this greater generality will be essential in \cref{appendix-completing-normal-proof}.\\ 

We begin by introducing some notation, and discussing how to profitably recast the problem of counting $N$-normal matrices as a problem of counting the number of solutions to an underdetermined system of linear equations. 
Given any matrix $X$, we let $r_i(X)$ and $c_i(X)$ denote its $i^{th}$ row and column respectively. With this notation, note that for a given matrix $M$, being $N$-normal is equivalent to satisfying the following equation for all $i,j \in [n]$:
\begin{equation}\label{equivalent to normal} r_i(M)r_j^T(M)-c_i(M)^Tc_j(M)=N_{ij}.
\end{equation}
In particular, writing $M$ in block form as:
$$M=\left[\begin{array}{cc} A_k & B_k\\
                            C_k & D_k\end{array}\right],$$
where $A_k$ is a $k\times k$ matrix, we see that \eqref{equivalent to normal} amounts to the following equations: 
\begin{enumerate}[$(i)$]
  \item For all $i,j\in[k]$:  $$r_i(A_k)r_j(A_k)^T+r_i(B_k)r_j(B_k)^T-c_i(A_k)^Tc_j(A_k)-c_i(C_k)^Tc_j(C_k)=N_{ij}.$$
\item For all $i\in[k],j\in[n-k]$: 
$$r_i(A_k)r_j(C_k)^T+r_i(B_k)r_j(D_k)^T-c_i(A_k)^Tc_j(B_k)-c_i(C_k)^Tc_j(D_k)=N_{i,k+j}.$$
\item  For all $i,j\in[n-k]$: $$r_i(C_k)r_j(C_k)^T+r_i(D_k)r_j(D_k)^T-c_i(B_k)^Tc_j(B_k)-c_i(D_k)^Tc_j(D_k)=N_{k+i,k+j}.$$
\end{enumerate}

We now rewrite this system of equations in a form that will be useful for our application. Following Deneanu and Vu, we will count the size of $\mathcal{N}(N)$ by constructing $N$-normal matrices in $n+1$ steps, and bounding the number of choices available at each step. The steps are as follows: in Step $0$, we select $n$ entries $d_1,\dots,d_n$ to serve as diagonal entries of the matrix $M$; in Step $k$ for $1\leq k\leq n$, we select $2(n-k)$ entries so as to completely determine the $k^{th}$ row and the $k^{th}$ column of $M$ -- of course, these $2(n-k)$ entries cannot be chosen arbitrarily, and must satisfy some constraints coming from the choice of entries in Steps $0,\dots,k-1$. 

More precisely, let $M_k$ denote the structure obtained at the end of Step $k$. Then,
\begin{equation}
\label{eqn:M_k}
M_k=\left[\begin{array}{c|c}A_k & B_k\\ \hline 
C_k & \begin{array}{ccc}d_{k+1} & * &*\\
 *&\ddots&* \\ *&*&d_n \\ \end{array} \end{array}\right],
\end{equation}
where the $*$'s denote the parts of $D_k$ which have not been determined by the end of Step $k$. Observe that the matrix $A_k$, together with the first column of $B_k$, the first row of $C_k$, and the diagonal element $d_{k+1}$ forms the matrix $A_{k+1}$; in particular, the matrix $A_{k+1}$ is already determined at the end of Step $k$. Moreover, both $B_{k+1}$ and $C_{k+1}$ are determined at the end of Step $k$ up to their last row and last column respectively. 

In Step $k+1$, we choose $r_{k+1}(B_{k+1})$ and $c_{k+1}(C_{k+1})$. In order to make this choice in a manner such that the resulting $M_{k+1}$ admits even a single extension to an $N$-normal matrix, it is necessary that for all $i\in[k]$: 
$$r_{k+1}(A_{k+1})r_i(A_{k+1})^T+r_{k+1}(B_{k+1})r_{i}^T(B_{k+1})-c_{k+1}(A_{k+1})^Tc_i(A_{k+1})-c_{k+1}(C_{k+1})^Tc_i(C_{k+1})=N_{k+1,i}.$$

Since $A_{k+1}$ is completely determined by the end of Step $k$, and since $N$ is fixed, we can rewrite the above equation as: for all $i\in [k]$,
\begin{equation}\label{N depends on previous}
r_{k+1}(B_{k+1})r_{i}^T(B_{k+1})-c_{k+1}(C_{k+1})^Tc_i(C_{k+1})=N'_{k+1,i},
\end{equation}
for some $N'_{k+1,i}$ which is uniquely determined at the end of Step $k$. Let $N'_{k}$ be the $k$-dimensional column vector whose $i^{th}$ entry is given by $N'_{k+1,i}$, let $T_{k}:=[U\text{ } V]$ be the $k\times 2(n-k-1)$ matrix formed by taking $U$ to be the matrix consisting of the first $k$ rows of $B_{k+1}$ and $V^{T}$ to be the matrix consisting of the first $k$ columns of $C_{k+1}$, and let $x_k$ be the $2(n-k-1)$-dimensional column vector given by $x_k:=\left[\begin{array}{c} r_{k+1}^T(B_{k+1})\\
-c_{k+1}(C_{k+1}) \end{array}\right]$. With this notation, \cref{N depends on previous} can be written as: 
\begin{equation} \label{key linear equation} T_kx_k=N'_k.
\end{equation} 

The next proposition is the analogue of \cref{prop:rank-partition-hadamard} in the present setting. 
\begin{proposition}\label{lemma:large rank for many}
Let $0 < \gamma < 1$ be fixed, and let $M$ be a random $m\times n'$ $\{\pm 1\}$-valued random matrix. Let $\mathcal{E}_{\gamma,\ell}$ denote the event that $M$ does not admit a $(\gamma m,\ell)$-rank partition. Then,
  $$\Pr[\mathcal{E}_{\gamma,\ell}]\leq 2^{-(1-\gamma)^{2}mn'+(1-\gamma)^{2}(m^{2}\ell+m^{2})+O(n')}.$$
\end{proposition}
The proof of this proposition is based on the following lemma, which follows easily from Odlyzko's lemma (\cref{odlyzko}). 

\begin{lemma}\label{lemma: large rank for 1}
Let $0 < \gamma < 1$ be fixed, and let $M$ be a random $m\times m$ $\{\pm 1\}$-valued random matrix. Then,
$$\Pr[\rank(M)\leq \gamma m]\leq 2^{-(1-\gamma)^2m^2+O(m)}.$$
\end{lemma}

\begin{proof}
For any integer $1\leq s\leq m$, let $\mathcal{R}_s$ denote the event that $\rank(M)=s$. Since 
$$\Pr[\rank(M)\leq \gamma m] = \Pr\left[\bigvee_{s=1}^{m}\mathcal{R}_{s}\right]\leq \sum_{s=1}^{\gamma m}\Pr[\mathcal{R}_s],$$
it suffices to show that $\Pr[\mathcal{R}_s]\leq 2^{-(1-\gamma)^{2}m^{2}+O(m)} $ for all $s\in [\gamma m]$. To see this, note by symmetry that 
$$\Pr[\mathcal{R}_s]\leq {m \choose s}\Pr\left[\mathcal{R}_s \wedge \mathcal{I}_{[s]}\right],$$
where $\mathcal{I}_{[s]}$ is the event that the first $s$ rows of $M$ are linearly independent. Moreover, letting $r_{[s+1,n]}(M)$ denote the set $\{r_{s+1}(M),\dots,r_m(M)\}$ of the last $m-s$ rows of $M$, and $V_s$ denote the random vector space spanned by the first $s$ rows of $M$, we have: 
\begin{eqnarray*}
\Pr\left[\mathcal{R}_{s}\wedge\mathcal{I}_{[s]}\right] & \leq & \Pr\left[r_{[s+1,n]}\subseteq V_s\right]\\
 & = & \sum_{v_{1},\dots,v_{s}\in\{\pm1\}^{m}}\Pr\left[r_{[s+1,n]}(M)\subseteq V_s|r_{i}(M)=v_{i},1\leq i\leq s\right]\Pr\left[r_{i}(M)=v_{i},1\leq i\leq s\right]\\
 & = & \sum_{v_{1},\dots,v_{s}\in\{\pm1\}^{m}}\left(\prod_{j=s+1}^{m}\Pr\left[r_{j}(M)\subseteq V_s|r_{i}(M)=v_{i},1\leq i\leq s\right]\right)\Pr\left[r_{i}(M)=v_{i},1\leq i\leq s\right]\\
 & \leq & \sum_{v_{1},\dots,v_{s}\in\{\pm1\}^{m}}2^{(s-m)(m-s)}\Pr\left[r_{i}(M)=v_{i},1\leq i\leq s\right]\\
 & = & 2^{-(m-s)^{2}},
\end{eqnarray*}
where the second line follows from the law of total probability; the third line follows from the independence of the rows of the matrix $M$; and the fourth line follows from Odlyzko's lemma (\cref{odlyzko}) along with the fact that conditioning on the values of $r_1(M),\dots,r_s(M)$ fixes $V_s$ to be a subspace of dimension at most $s$.   

Finally, since $m-s\geq (1-\gamma)m$ and ${m \choose s} \leq 2^{m}$, we get the desired conclusion.  
\end{proof}

\begin{proof}[Proof of \cref{lemma:large rank for many}] 
For each $i\in [t]$, where $t = \lfloor n'/m\rfloor$, let $A_i$ denote the $m\times m$ submatrix of $M$ consisting of the columns $c_{(i-1)m+1}(M),\dots,c_{im}(M)$. Then, 
$$\Pr[\mathcal{E}_{\gamma,\ell}] \leq \Pr\left[|\{i\in [t]:\rank(A_i)\leq \gamma m\}|> t-\ell \right].$$
By \cref{lemma: large rank for 1}, we have for each $i\in [t]$ that
$$\Pr[\rank(A_i)\leq \gamma m]\leq 2^{-(1-\gamma)^2m^2+O(m)}.$$
Therefore, since the entries of the different $A_i$'s are independent, the probability of having more than $t-\ell$ indices $i\in [t]$ for which $\rank(A_i)\leq \gamma m$ is at most:
\begin{eqnarray*}
\sum_{k=t-\ell+1}^{t}{t \choose k}\left(2^{-(1-\gamma)^{2}m^{2}+O(m)}\right)^{k} & \leq & t2^{t}2^{-(1-\gamma)^{2}m^{2}(t-\ell)+O(tm)}\\
 & \leq & 2^{-(1-\gamma)^{2}m^{2}t+(1-\gamma)^{2}m^{2}\ell+O(tm)}\\
 & \leq & 2^{-(1-\gamma)^{2}mn'+(1-\gamma)^{2}(m^{2}\ell+m^{2})+O(n')},
\end{eqnarray*}
which completes the proof.
\end{proof}

We need one final piece of notation. For $1\leq k \leq n$, we define the set of $k$-partial matrices -- denoted by $\mathcal{P}_{k}$ -- to be $\{\pm 1, \ast\}$-valued matrices of the form \cref{eqn:M_k}. For any $n\times n$ $\{\pm1\}$-valued matrix $M$, let $M_k$ denote $k$-partial matrix obtained by restricting $M$. For any $1\leq k \leq n$ and any $n\times n$ matrix $N$, we define: 
$$\mathcal{S}_{k}(N):= \{P\in \mathcal{P}_k:P=M_k\text{ for some }M\text{ which is }N\text{-normal}\}.$$
In words, $\mathcal{S}_k(N)$ denotes all the possible $k$-partial matrices arising as restrictions of $N$-normal matrices. The following proposition is the main result of this section. 
\begin{proposition}
\label{prop-normal-small-improvement}
There exist absolute constants $\beta,\delta>0$ such that for any $n\times n$ matrix $N$, 
$$|\mathcal{S}_{\beta n}(N)|\leq 2^{(2\beta-\beta^2)n^2-\delta n^2+o(n^2)}.$$ 
\end{proposition}
Given this proposition, it is immediate to obtain a non-trivial upper bound on the number of $\{\pm 1\}$-valued $N$-normal matrices. Indeed, any $N$-normal matrix must be an extension of a matrix in $\mathcal{S}_{\beta n}(N)$; on the other hand, any matrix in $\mathcal{S}_{\beta n}(N)$ can be extended to at most $2^{(1-\beta)^2n^2}$ $N$-normal matrices (as $D_{\beta n}$ is an $n(1-\beta)\times n(1-\beta)$ $\{\pm 1\}$-valued matrix). Hence, the number of $N$-normal matrices is at most $2^{(2\beta-\beta^2)n^2-\delta n^2+(1-\beta)^2n^2}=2^{(1-\delta)n^2+o(n^2)}$.  
\begin{proof}
For any $m$-partial matrix $P$ and for any $1\leq k \leq m$, let $T_k(P)$ denote the $k \times 2(n-k-1)$ matrix obtained from $P$ as in \cref{key linear equation}. We will estimate the size of $\mathcal{S}_{\beta n}(N)$ by considering the following two cases. 

First, we bound the number of partial matrices $P$ in $\mathcal{P}_{\beta n}$ such that for some $\beta n/2 \leq k \leq \beta n$, $T_{k}(P)$ does not admit a $(\gamma k,\ell_k)$-rank-partition, where $\ell_k=n'/2k, n'=2(n-k-1)$, and $0<\gamma <1$ is some constant to be chosen later. For this, note that \cref{lemma:large rank for many} shows that there are at most 
$$2^{kn'-(1-\gamma)^2kn'+(1-\gamma)^2(k^2\ell_k + k^{2})+O(n')}=2^{kn'-(1-\gamma)^{2}kn'/4+O(n')}$$ choices for such a $T_k(P)$, provided $k < n'/4$, which holds for (say) $\beta < 1/4$. Since the remaining unknown entries of $P$ which are not in $T_{k}(P)$ are $\{\pm 1\}$-valued, this shows that the number of $\beta n$ partial matrices satisfying this first case is bounded above by 
$$2^{(2\beta-\beta^2)n^2-n^{2}(1-\gamma)^{2}\beta(1-\beta)/4+o(n^2)},$$
for all $\beta < 1/4$. 
 
Second, we bound the number of partial matrices $P \in \mathcal{S}_{\beta n}(N)$ which have the additional property that $T_k(P)$ admits a $(\gamma k, \ell_k)$-rank-partition for all $\beta n/2 \leq k \leq \beta n$. In this case, \cref{thm:halasz-int} shows that for any $\beta n/2 \leq k \leq \beta n$, the number of $\{\pm 1\}$-valued solutions 
to \cref{key linear equation} $T_k(P)x_k=N'$
is at most 
\begin{equation} \label{eq:estimate for number of solutions in one step}2^{2(n-k-1)}\ell_k^{-\gamma k/2}\leq 2^{2(n-k)-\frac{\gamma k}{2}\log_{2}\frac{n}{2k}},\end{equation}
where in the last inequality, we have used $2(n-k-1)\geq n$ for all $k\leq \beta n$, which is certainly true for $\beta < 1/4$.  
In other words, for a fixed $T_k(P)$, there are at most $2^{2(n-k-1)-\frac{\gamma k}{2}\log_{2}\frac{n}{2k}}$ ways to extend it to $T_{k+1}(P')$ for some $P'\in \mathcal{S}_{\beta n}(N)$. Hence, it follows that the number of matrices in $\mathcal{S}_{\beta n}(N)$ with this additional property (stated at the beginning of the paragraph) is at most:
  $$2^{(2\beta-\beta^2)n^2-\sum_{k=\beta n/2}^{\beta n}\frac{\gamma k}{2}\log \frac{n}{2k}}\leq 2^{(2\beta-\beta^2)n^2-n^{2}\gamma \beta^{2}\log_{2}(1/2\beta)/8+o(n^2)},$$
for $\beta < 1/4$.    
Combining these two cases completes the proof. 
\end{proof}
\begin{remark}
\label{rmk:value-of-delta}
In particular, if we take $\gamma = 3/4$, it follows that for $\beta$ sufficiently small (say $\beta \leq 2^{-10}$), we can take $\delta \geq \beta^{2}$.  
\end{remark}

\section{Acknowledgements}
V.J. would like to thank Ethan Yale Jaffe for his insightful comments. A.F. and Y.Z. would like to thank Gwen McKinley, Guy Moshkovitz, and Clara Shikhelman for helpful discussions at the initial stage of this project.
\bibliographystyle{plain}
\bibliography{hadamard}

\appendix
\section{Completing the proof of \cref{thm:normal}}
\label{appendix-completing-normal-proof}
We now show how to combine the strategy of Deneanu and Vu with \cref{subsec-proof-normality} in order to prove \cref{thm:normal}. We begin with a few definitions. 

\begin{definition}
Let $S_n$ denote the symmetric group on $n$ letters. For any $\sigma\in S_n$ and for any $n\times n$ matrix $M$, we define 
$$M_{\sigma}:=P_{\sigma}MP_{\sigma}^T,$$
where $P_{\sigma}$ is the permutation matrix representing $\sigma$. In other words, $M_{\sigma}$ is the matrix obtained from $M$ by permuting the row and columns according to $\sigma$. 
\end{definition}

The previous definition motivates the following equivalence relation $\sim$ on the set of $n\times n$ matrices: given two $n\times n$ matrices $M$ and $M'$, we say that $M \sim M'$ if and only if there exists $\sigma \in S_n$ such that $M' = M_\sigma$. The next definition isolates a notion of normality which is invariant under this equivalence relation.     

\begin{definition}
Let $N$ be a fixed $n\times n$ matrix. We say that an $n\times n$ matrix $M$ is $N$-\emph{normal-equivalent} if and only if there exists some $\sigma\in S_n$ such that $MM^T-M^TM=N_{\sigma}.$
\end{definition}
By definition, it is clear that for any $N$ and for any $M \sim M'$, $M$ is $N$-normal-equivalent if and only if $M'$ is $N$-normal equivalent. On the other hand, as we will see below, one can find a permutation $\rho_{M}$ for any matrix $M$ such that for the matrix $M':=M_{\rho_{M}}$, the ranks of many of the matrices $T_k(M'), 1\leq k\leq n$ are large, where $T_k(M')$ denotes the matrix from \cref{key linear equation}. Therefore, by Odlyzko's lemma, we will be able to obtain good upper bounds on the probability of the random matrix $M':=M_{\rho_M}$ being $C$-normal, for any fixed $C$, which then translates to an upper bound on the probability of $N$-normality of $M$ as follows: for any fixed $N$,
\begin{eqnarray*}
\Pr\left[M\text{ is }N\text{-normal}\right] & \leq & \Pr\left[M\text{ is \ensuremath{N}-normal-equivalent}\right]\\
 & = & \Pr\left[M_{\rho_{M}}\text{ is \ensuremath{N}-normal-equivalent}\right]\\
 & \leq & \sum_{\sigma\in S_{n}}\Pr\left[M_{\rho_{M}}\text{ is $N_{\sigma}$-normal}\right]\\
 & \leq & n!\sup_{\sigma\in S_{n}}\Pr\left[M_{\rho_{M}}\text{ is $N_{\sigma}$-normal}\right]\\
 & \leq & 2^{o(n^{2})}\sup_{C\in\mathcal{M}_{n\times n}}\Pr\left[M_{\rho_{M}}\text{ is \ensuremath{C}-normal}\right],
\end{eqnarray*}
where $\mathcal{M}_{n\times n}$ denotes the set of all $n\times n$ matrices, and we have used the fact that $n!=2^{o(n^2)}$. Hence, it suffices to provide a good \emph{uniform} upper bound on the probability that the random matrix $M_{\rho_M}$ is $N$-normal for any fixed $N$. 

To make the special property of the matrix $M_{\rho_M}$ precise, we need the following functions, defined for all integers $1\leq s\leq t\leq n$:

$$R_{s,t}(i):=\left\{\begin{array}{cc} i & \text{ if }0<i\leq s\\
s &\text{ if }s<i\leq t\\
s+t-i &\text{ if }t<i\leq 2n-s-t\\
2n-2i & \text{ if }2n-s-t<i\leq n. \end{array}\right.$$

The next proposition is one of the key ideas in the proof of Deneanu and Vu. 
\begin{proposition}[Permutation Lemma, Lemma 3.5 in \cite{deneanu2017random}] Let $M$ be any (fixed) $n\times n$ matrix. Then, there exist $s,t\in \mathbb{N}$ and $\rho_M\in S_n$ such that $M_{\rho_M}$ satisfies:
  $$\rank(T_i(M_{\rho_M}))=R_{s,t}(i) \text{ for all }1\leq i\leq n.$$
\end{proposition}
For a fixed matrix $N$, let $\mathcal{N}_{s,t}(N)$ denote the set of $\{\pm 1\}$-valued $n\times n$ matrices $M$ such that $M$ is $N$-normal, and $\rank(T_i(M)) = R_{s,t}(i)$ for all $i\in [n]$. Then, it follows from the previous proposition that $M_{\rho_{M}}$ is $N$-normal if and only if $M_{\rho_M} \in \bigcup_{1\leq s\leq t\leq n}\mathcal{N}_{s,t}(N)$. This, in turn, can happen only if $M$ itself is one of the at most $n!\sum_{s,t}|\mathcal{N}_{s,t}(N)|$ matrices obtained by permuting the rows and columns of $\mathcal{N}_{s,t}(N)$. Hence, it suffices to provide a good upper bound on $|\mathcal{N}_{s,t}(N)|$ \emph{uniformly} in $N, s$ and $t$. 

Deneanu and Vu note (Observation 3.7 in \cite{deneanu2017random}) that $\mathcal{N}_{s,t}$ is empty unless the following restrictions on $s$ and $t$ are met: 
\begin{itemize}
\item $1\leq s \leq 2n/3$, and
\item $\frac{s}{2} < n-t < s$.
\end{itemize}
Then, letting 
$$\beta := \sup\{c>0: |\mathcal{N}_{s,t}(N)|\leq 2^{-(c+o(1))n^{2}} \text{ for all } s,t,N\},$$
and for some small fixed (but otherwise arbitrary) $\epsilon > 0$, letting
$$\alpha := \beta -\epsilon,$$
they show (Lemmas 5.1 and 5.4 in \cite{deneanu2017random}) the following:
\begin{equation}
\label{eqn:bounds}
|\mathcal{N}_{s,t}(N)|\leq2^{n^{2}}\times\begin{cases}
\min\left(2^{g_{1}(n,s,t)+o(n^{2})},2^{f(\alpha,n,s,t)+o(n^{2})}\right) & s\leq\frac{n}{2}\\
\min\left(2^{g_{2}(n,k,t)+o(n^{2})},2^{f(\alpha,n,k,t)+o(n^{2})}\right) & s\geq\frac{n}{2}
\end{cases},
\end{equation}
where 
\begin{eqnarray*}
f(\alpha,n,s,t) & := & (1-\alpha)t^{2}-s^{2}/2-n^{2}+ns\\
g_{1}(n,s,t) & := & t^{2}-3s^{2}+2sn+st-2nt\\
g_{2}(n,s,t) & := & n^{2}+s^{2}+t^{2}+st-2sn-2nt.
\end{eqnarray*}
Finally, they analyze \cref{eqn:bounds} to obtain their bound on $\beta$. For this, they note that since for fixed $s$, both $g_1$ and $g_2$ are decreasing functions of $t$ while $f$ is an increasing function of $t$, the worst restrictions on $\beta$ (i.e. those requiring $\beta$ to be small) can only be obtained in one of the following six cases: 
\begin{enumerate}
\item $t = n-s$ and $s\leq n/2$, which places the restriction $\beta \leq 0.425$;
\item $t = n-s/2$ and $s \leq n/2$, which places the restriction $\beta \leq 0.307$;
\item $t = n-s/2$ and $s \geq n/2$, which places the restriction $\beta \leq 0.3125$;
\item $t=s$ and $s \geq n/2$, which places the restriction $\beta \leq 0.323$;
\item $f(\alpha,n,s,t) = g_2(n,s,t)$ and $s \geq n/2$, which places the restriction $\beta \leq 0.307$; and finally, 
\item $f(\alpha,n,s,t) = g_1(n,s,t)$ and $s \leq n/2$, which places the worst restriction $\beta \leq 0.302$.
\end{enumerate}
Hence, any improvement in Case 6 translates to an overall improvement in their bound. Moreover, note that for $1\leq s\leq\frac{n}{10}$,
Case 6 only leads to the restriction $\beta\leq0.7$. Therefore, it
suffices to improve Case 6 for $\frac{n}{10}\leq s\leq\frac{n}{2}$.
We will do this using \cref{prop-normal-small-improvement}. 

We start by showing how to deduce the upper bound $g_{1}(n,s,t)$, as in \cite{deneanu2017random}. For any
$0\leq k\leq n$, we define 
\[
\mathcal{S}_{k,(s,t)}(N):=\{P\in\mathcal{P}_{k}:P=M_{k}\text{ for some }M\in\mathcal{N}_{s,t}(N)\},
\]
where recall that $\mathcal{P}_{k}$ denotes the set of $k$-partial
matrices, and $M_{k}$ denotes the $k$-partial matrix associated to $M$. By definition, the number of ways to extend
any $k$-partial matrix in $\mathcal{S}_{k,(s,t)}(N)$ to a $(k+1)$-partial
matrix in $\mathcal{S}_{k+1,(s,t)}(N)$ is at most the number of $\{\pm1\}$-valued
solutions to : $T_{k}x_{k}=N'_{k}$, which is at most $2^{\max\{2(n-k-1)-\rank(T_{k}),0\}}$ by Odlyzko's lemma (\cref{odlyzko}). 
Hence, it follows that the total number of matrices in $\mathcal{N}_{s,t}(N)$
is at most 
\begin{eqnarray*}
|\mathcal{N}_{s,t}(N)| & \leq & 2^{o(n^{2})}\prod_{k=0}^{n}2^{\max\{2(n-k-1)-\rank(T_{k}),0\}}\\
 & = & 2^{o(n^{2})}\prod_{k=0}^{n}2^{\max\{2(n-k)-R_{s,t}(i),0\}}\\
 & = & 2^{g_{1}(n,s,t)+o(n^{2})},
\end{eqnarray*}
where the second equality follows from the definition of $\mathcal{N}_{s,t}(N)$
and the last equality follows by direct computation. 

To obtain our improvement, we note that above computation may be viewed
in the following two steps:
\begin{itemize}
\item $|\mathcal{N}_{s,t}(N)| \leq 2^{o(n^{2})}|\mathcal{S}_{\beta n,(s,t)}(N)|\prod_{k=\beta n+1}^{n}2^{\max\{2(n-k-1)-\rank(T_{k}),0\}}$, which is true for any $0< \beta < 1$
\item For $0< \beta < 1$ such that $\beta n < s$, 
\begin{eqnarray*}
|\mathcal{S}_{\beta n,(s,t)}(N)| & \leq & 2^{o(n^{2})}\prod_{k=0}^{\beta n}2^{\max\{2(n-k-1)-\rank(T_{k}),0\}}\\
 & = & 2^{o(n^{2})}\prod_{k=0}^{\beta n}2^{2(n-k)-\rank(T_{k})}\\
 & = & 2^{o(n^{2})}\prod_{k=0}^{\beta n}2^{2n-3k}\\
 & = & 2^{(2\beta-\beta^{2})n^{2}-\beta^{2}n^{2}/2+o(n^{2})}
\end{eqnarray*}
In particular, by our assumption on $s$, we know that this holds for (say) $\beta = 2^{-10}$. 
\end{itemize}
However, by \cref{prop-normal-small-improvement} and \cref{rmk:value-of-delta}, we already know that for $\beta = 2^{-10}$,
\begin{eqnarray*}
|\mathcal{S}_{\beta n,(s,t)}(N)| & \leq & |\mathcal{S}_{\beta n}(N)|\\
 & \leq & 2^{(2\beta - \beta^{2})n^{2} - \beta^{2}n^{2}+o(n^{2})}
\end{eqnarray*}
Using this improved bound in the previous computation, we get that
$|\mathcal{N}_{s,t}(N)|\leq2^{h(n,s,t)+o(n^{2})}$, where 
\[
h(n,s,t)=g_1(n,s,t)- \frac{\beta^{2}n^{2}}{2}.
\]
Hence, we have showed that Case 6 can be replaced by the following two cases:
\begin{enumerate}[\text{Case} 6.1]
\item $f(\alpha,n,s,t) = g_1(n,s,t)$ and $s\leq n/10$
\item $f(\alpha,n,s,t) = h(n,s,t)$ and $n/10 \leq s \leq n/2$,
\end{enumerate}
each of which place a restriction on $\beta$ which must be larger than the constant $c_{DV}$ obtained in \cite{deneanu2017random}. This completes the proof of \cref{thm:normal}.

\end{document}